\documentclass{amsart}[11pt]
\usepackage{enumerate}
\usepackage{amsmath}
\usepackage{a4wide}
\usepackage[english]{babel}

\usepackage[mathscr]{euscript}
\usepackage{epsfig}
\usepackage{latexsym}

\usepackage{amssymb}
\usepackage{amsthm}
\usepackage{amsmath}
\usepackage{amsxtra}

.tfm
.tfm

\theoremstyle{plain}
\newtheorem{prop}{Proposition}[section]

\newtheorem*{prop*}{Proposition}
\newtheorem{thm}[prop]{Theorem}
\newtheorem*{thm*}{Theorem}
\newtheorem{coro}[prop]{Corollary}

\newtheorem{lemma}[prop]{Lemma}

\theoremstyle{definition}
\newtheorem*{defi}{Definition}

\theoremstyle{remark}

\newtheorem*{remark}{Remark}

\numberwithin{table}{section}

\DeclareMathOperator{\Gal}{Gal}
\DeclareMathOperator{\Hom}{Hom}

\DeclareMathOperator{\coho}{H}

\DeclareMathOperator{\sha}{III}

\newcommand{\F}{\mathbb F}

\newcommand{\HH}{\mathcal H}

\newcommand{\Om}{{\mathscr{O}}}

\newcommand{\nequiv}{\not\equiv}

\def\<#1>{{\left\langle{#1}\right\rangle}}

\def\Z{{\mathbb Z}}             
\def\Q{{\mathbb Q}}             

\def\set#1{{\left\{{\def\st{\;:\;}#1}\right\}}}

\def\O{{\mathcal{O}}}           
\def\Ol(#1){{\mathop{\O_l}(\id{#1})}}                 
\def\Or(#1){{\mathop{\O_r}(\id{#1})}}                 
\def\Orp(#1){{\mathop{\O_r}(\idp{#1})}}               
\def\Otern(#1){{\mathop{\O^0_r}(\id{a})}}    

\def\id#1{{\mathfrak{#1}}}      
\def\idp#1{{\id{#1}_p}}         

\def\T_#1(#2){{\mathop{\mathscr T}\nolimits_{#1}(\id{#2})}}
\def\TO(#1)_#2(#3){{\mathop{\mathscr T}\nolimits^{#1}_{#2}(\id{#3})}}

\def\A_#1(#2){{\mathop{\mathscr A}\nolimits_{#1}(\id{#2})}}
\def\Ax_#1(#2){{\mathop{\widetilde{\mathscr A}}\nolimits_{#1}(\id{#2})}}

\def\Ad{Ad^0\bar{\rho}}

\def\Adn{Ad^0\rho_n}
\def\rrho{\overline{\rho}}

\def\trho{\tilde{\rho}}
\def\T{\mathcal{T}}
\DeclareMathOperator{\GL}{GL}
\DeclareMathOperator{\SL}{SL}
\DeclareMathOperator{\PSL}{PSL}
\DeclareMathOperator{\PGL}{PGL}
\DeclareMathOperator{\Img}{Im}
\DeclareMathOperator{\Ind}{Ind}
\DeclareMathOperator{\Ker}{Ker}

\usepackage[latin1]{inputenc}
\usepackage[all]{xy}
\usepackage{color}
\usepackage[active]{srcltx}
\usepackage{url}
\begin{document}

\title{Lifting Galois representations to ramified coefficient fields}
\author{Maximiliano Camporino}
\address{Departamento de Matem\'atica, Facultad de Ciencias Exactas y Naturales, Universidad de Buenos Aires} 
\email{mcamporino@dm.uba.ar}
\thanks{The author was partially supported by a CONICET doctoral fellowship}

\begin{abstract}
  Let $p>5$ be a prime integer and $K/\Q_p$ a finite ramified
  extension with ring of integers $\Om$ and uniformizer $\pi$. Let $n
  >1$ be a positive integer and $\rho_n:G_\Q \to \GL_2(\Om/\pi^n)$ be
  a continuous Galois representation. In this article we prove that
  under some technical hypotheses the representation $\rho_n$ can be
  lifted to a representation $\rho:G_\Q \to \GL_2(\Om)$. Furthermore,
  we can pick the lift restriction to inertia at any finite set of
  primes (at the cost of allowing some extra ramification) and get a
  deformation problem whose universal ring is isomorphic to
  $W(\F)[[X]]$. The lifts constructed are ``nearly ordinary'' (not necessarily
  Hodge-Tate) but we can prove the existence of ordinary modular
  points (up to twist).
\end{abstract}

\maketitle

\section{Introduction}

The present article is a continuation of the work done in
\cite{arxiv}, where we constructed, for a finite field $\F$, lifts of
representations $\rho_n:G_\Q \to \GL_2(W(\F)/p^n)$ to
$\GL_2(W(\F))$. Here we prove how to extend the results to finite
ramified extensions $K/\Q_p$ of ramification degree $e$.

The method used in \cite{arxiv} followed the ideas of \cite{ravi2} and
\cite{ravi3}, adapted to the modulo $p^n$ setting. As noticed in
\cite{arxiv} (see the remark before Proposition 5.9) these methods do
not generalize to representations $\rho_n:G_\Q \to \GL_2(\Om/\pi^n)$
where $\Om$ is the ring of integers of $K$. 

The obstacle is that the modulo $\pi^2$ reduction of $\rho_n$ (which we
denote $\rho_2$) fixes the same field extension of $\Q$ as an element
$f \in \coho^1(G_\Q,\Ad)$ (where $\Ad$ isthe adjoint representation of 
the reduction mod $\pi$ of $\rho_n$). This implies that whenever we define
a local condition for deformations containing $\rho_n$ at a prime $v$ it
automatically contains $\rho_2$ and therefore $f$ lies inside its tangent space.
On this way, no matter which set of primes $M$ we choose, the morphism of Theorem $A$ of 
\cite{arxvv}

\begin{equation}
\label{eq:iso}
\coho^1(G_M,\Ad) \to \bigoplus_{v \in M} \coho^1(G_v,\Ad)/N_v
\end{equation}

will always have non-trivial kernel and therefore will never be an isomorphism (as required in \cite{arxiv}).


The key innovation of this work is to relax local conditions so that the morphism
(\ref{eq:iso}) is no longer an isomorphism, but a surjective map with
one dimensional kernel.  This will be enough for the lifting purpose,
since it allows to find global elements that make the required the local adjustments at each
step. We cannot relax conditions at primes $v \neq p$, since the
local deformation ring of $\rrho|_{G_v}$ should not have a smooth
quotient of dimension bigger than $\dim \coho^1(G_v,\Ad) - \dim
\coho^2(G_v,\Ad)$. Therefore, we need to impose a different local
condition at the prime $p$. The condition we impose is the same as in \cite{CM}.

\begin{defi}
 We say that a deformation is ``nearly ordinary'' if its restriction to the inertia subgroup is upper-triangular and its 
 semisimplification is not scalar, i.e. if
 $$\rho|_{I_p} = \begin{pmatrix}
                  \psi_1&*\\
                  0&\psi_2
                 \end{pmatrix}.
 $$ with $\psi_1 \neq \psi_2$.
\end{defi}

Using this local condition at $p$, we are able to derive a slightly weaker version of the following theorem (see Theorem \ref{abstract}
for the precise statement), which is one of the main results of this work. 

\begin{thm*}
  Let $\rho_n: G_\Q \to \GL_2(\Om/\pi^n)$ be a continuous
  representation which is odd and nearly ordinary at $p$.  Assume that
  $\Img(\rho_n)$ contains $\SL_2(\Om/p)$ if $n \ge e$ and $\SL_2(\Om/\pi^n)$ otherwise. Let $P$ be a set of primes of $\Q$ containing
  the ramification set of $\rho_n$.  For each $v \in P \backslash
  \set{p}$ fix a local deformation $\rho_v: G_v \to \GL_2(\Om)$ that
  lifts $\rho_n|_{G_v}$.  Then there exists a continuous
  representation $\rho: G_\Q \to \GL_2(\Om)$ and a finite set of
  primes $R$ such that:
 \begin{itemize}
  \item $\rho$ lifts $\rho_n$, i.e. $\rho \equiv \rho_n \pmod{\pi^n}$.
  \item $\rho$ is unramified outside $P \cup R$.
  \item For every $v\in P$, $\rho|_{I_v} \simeq \rho_v|_{I_v}$.
  \item $\rho$ is nearly ordinary at $p$.
  \item All the primes of $R$, except possibly one, are not congruent to $1$ modulo $p$.
 \end{itemize}
\end{thm*}

In fact, the method provides us not only a lift of $\rho_n$ to $\Om$ but a family of lifts to 
characteristic $0$ rings, parametrized by a lift to the coefficient ring $W(\F)[[X]]$ (see Theorem
\ref{abstract2}). The downside is that this family of representations
is not ordinary but nearly ordinary, which implies that
most points are not Hodge-Tate (in particular not modular). However,
the freedom of the coefficient ring allows us to prove the existence
of modular points, which is the second main result of the present
article (see Theorem~\ref{modular} for the precise statement).

\begin{thm*}
 Let $p$ be a prime, $\Om$ the ring of integers of a finite extension $K/\Q_p$ with ramification degree $e > 1$ and $\pi$ its local uniformizer.
 Let $\rho_n: G_\Q \to \GL_2(\Om/\pi^n)$ be a continuous representation satisfying
 \begin{itemize}
  \item $\rho_n$ is odd.
  \item $\Img(\rho_n)$ contains $\SL_2(\Om/p)$ if $n \ge e$ and $\SL_2(\Om/\pi^n)$  otherwise.
  \item $\rho_n$ is ordinary at $p$.
 \end{itemize}
 
 Let $P$ be a set of primes containing the ramification set of
 $\rho_n$, and for each $v \in P$ pick a local deformation
 $\rho_v:G_v\to \GL_2(\Om)$ lifting $\rho_n|_{G_v}$. Then there exists
 a finite set of primes $Q$ and a continuous representation $\rho:
 G_{P\cup Q} \to \GL_2(\Om)$ such that
 
 \begin{itemize}
  \item $\rho$ lifts $\rho_n$, i.e. $\rho \equiv \rho_n \pmod{\pi^n}$.
  \item $\rho$ is modular.
  \item For every $v \in P$, $\rho|_{I_v} \simeq \rho_v|_{I_v}$.
  \item For every $q \in Q$ $\rho|_{I_q}$ is unipotent and all but possibly one prime of $Q$ satisfy $q \neq \pm 1 \pmod{p}$ .
  \item $\rho$ is ordinary at $p$.
 \end{itemize}
\end{thm*}

The strategy to prove both theorems is similar to the one in
\cite{ravi3} and \cite{arxiv}. We will construct, for each prime $v\in
P$, a set of deformations of $\rho_n|_{G_v}$ to $\Om$ which contains
$\rho_v$ and a subspace $N_v$ preserving its reductions in the sense
of Proposition \ref{localtypes}. Also, for $v = p$, we take $C_p$ to
be the set of nearly ordinary deformations, which will give a larger
subspace $N_p$ and therefore a smaller codomain for the morphism (\ref{eq:iso}). 
Given this local setting, and after some
manipulation of the groups appearing in (\ref{eq:iso}) (and its
analogue for $\coho^2(G_T,\Ad)$) we are able to make such map
surjective (with a $1$-dimensional kernel) and the corresponding map 
for $\coho^2(G_T,\Ad)$ injective. This implies that, with enoguh local conditions,
the problem of lifting $\rrho$ is unobstructed, which gives Theorem \ref{abstract2}.

The tricky part here is that, differently from what happens in
\cite{arxiv}, in some cases the subspaces $N_v$ preserve the
reductions modulo $\pi^n$ of the elements of $C_v$, not for all $n$
but for $n$ bigger than a certain integer $\alpha$. To overcome this
situation we will, following the ideas of \cite{ravilarsen}, lift by
adding one set of auxilliary primes for each power of $\pi$, until we
reach the lift modulo $\pi^{\alpha}$ for the main method to work. In
this way we will get Theorem \ref{abstract}.

Theorem \ref{modular} follows from studying the possible modular points appearing in the universal ring provided by Theorem \ref{abstract2}.
Notice that we will obtain a modular lift of $\rho_n$ each time we find a nearly ordinary lift $\rho$ such that the characters appearing
in the diagonal of $\rho|_{I_p}$ can be written as an integral power of the cyclotomic character times a character of finite order.

We would like to remark that the method employed for the proof of Theorem \ref{abstract}
seems to generalize well to base fields other than $\Q$. However, the more interesting question of 
getting an analog to Theorem \ref{modular} seems to be of higher difficulty. We currently have 
some work in progress towards this direction. As a final remark, the results obtained in this
work have overlap with the ones in \cite{KR} (in particular Theorem \ref{abstract} is analog to Theorem $11$ of \cite{KR}). 
Both works were independent in their first versions and in the present one we applied some of their
results in order to remove hyphoteses. Also, the methods of section $5$ of \cite{KR} allow one
to get the stronger Corollary \ref{cor}.

The article is organized as follows: Section 2 concerns with the
construction of the sets $C_q$ and subspaces $N_q$ for primes $q\neq
p$. In Section 3 we do the same for the nearly ordinary condition at
$p$. Section 4 treats global arguments for lifting and
proves the first main theorem of this work. It is divided into two
subsections, one for exponents at which we have $C_q$ and $N_q$
for every $q\in P$ and one for the ones at which we do not. In Section
5 we prove the other main theorem which concerns modularity.

\subsection{Notation and conventions} In this article $p$ will denote
a rational prime, $\Om$ the ring of integers of a ramified finite
extension $K/\Q_p$, $\pi$ its local uniformizer and $e$ its
ramification degree. We will denote the residual field
$\Om/\pi$ by $\F$. For a prime $q$ we denote by $G_q$ the local absolute
Galois group $\Gal(\overline{\Q_q}/\Q_q)$ and $\sigma$ and $\tau$
stand for a Frobenius element and a generator of the tame inertia
group of $G_q$.  We denote
\[
e_1 = \begin{pmatrix} 1&0\\0&-1
                     \end{pmatrix}, \hspace{0.2cm} e_2 = \begin{pmatrix}
                      0&1\\0&0
                     \end{pmatrix}, \hspace{0.2cm} \text{and} \hspace{0.2cm} e_3 = \begin{pmatrix}
                      0&0\\1&0
                     \end{pmatrix}
\]
which form a basis for the space of $2\times 2$ matrices with trace
$0$. Also, given a prime $q$, $d_i$ will stand for $\dim \coho^i(G_q,\Ad)$ for $i = 0,1,2$.
The cyclotomic character will be denoted by $\chi$. All our deformations will have fixed 
determinant. By $\rho_n:G_\Q \to \GL_2(\Om/\pi^n)$ we denote a continuous Galois
representation, $\rrho:G_\Q \to \GL_2(\F)$ denotes its residual
representation, and by $\rho$ we denote a representation with image in
$\GL_2(\Om)$. Finally $v$ stands for the valuation in
$\Om$ satisfying $v(\pi) = 1$.

For the definitions and main results of deformation theory, we refer to 
\cite{mazur}.
                     
\subsection{Acknowledgments} We would like to thank Ariel Pacetti for
many useful conversations, and to Ravi Ramakrishna for many remarks
and suggestions in a draft version of the present article.

%
%

\section{Local deformation theory at $q\neq p$}

Let $q\neq p$ a prime and let $\rho: G_q \to \GL_2(\Om)$ be a continuous representation. We
denote by $\rrho$ its reduction mod $\pi$. We know that the elements of $\coho^1(G_q,\Ad)$
act on the deformations $\tilde{\rho}$ with coefficients in $\Om/\pi^m$ by $u\cdot\tilde{\rho} = (1+\pi^{m-1}u)\tilde{\rho}$. 
Recall the notion of an element of $\coho^1(G_q,\Ad)$ 
preserving a set of deformations.

\begin{defi}
  Let $C_q$ be a set of deformations $\{\rho:G_q \to \GL_2(\Om)\}$ and $u\in
  \coho^1(G_q,\Ad)$.  We say that $u$ preserves $C_q$ if for any
  $\tilde{\rho}: G_q \to \GL_2(\Om/\pi^m)$ which is the reduction of a
  deformation of $C_q$ we have that $u\cdot\tilde{\rho}$ is also
  the reduction of some deformation of $C_q$.
\end{defi}

We want to prove that for every $\rho$ there exists a set $C_q$ of
deformations containing it that is preserved by a subspace $N_q$ of
certain dimension. We can prove such set exists for almost all
possible $\rho$ but some particular ones.

\begin{defi}
\label{bad}
We say that a representation $\rho:G_q \to \GL_2(\Om)$ is ``bad''
if 
\[
\rho \simeq \begin{pmatrix}
  \psi_1& \frac{\psi_1-\psi_2}{\pi^r}\\
  0&\psi_2
\end{pmatrix}
\]
 over $\GL_2(\Om)$ with $r>0$ and moreover the following holds:
 \begin{itemize}
  \item $\rrho$ is unramified and $\rrho^{ss}(\sigma)$ is scalar.
  \item $v(\psi_1(\sigma) - \psi_2(\sigma)) < v(\frac{\psi_1(\tau) - \psi_2(\tau)}{\pi^r})$.
 \end{itemize}
\end{defi}

We currently do not know how to find the desired family of deformations in these cases. Observe that finding the families $C_q$ and
subspaces $N_q$ accounts for proving that the Balancedness Assumption of \cite{KR} holds (see Assumption before Definition 3). In that 
scenario, if we start with a mod $\pi^n$ deformation $\rho_n$ and a prime $q$ for which every lift to $\Om$ is bad, we do not know 
how to prove that the Assuption holds.

\begin{prop}
\label{localtypes}
Let $\rho:G_q \to \GL_2(\Om)$ be a continuous representation. If $\rho$ is not bad then there always exists a positive integer
$\alpha$, a set $C_q$ of deformations of the reduction of $\rho$ modulo $\pi^\alpha$ to characteristic $0$ that
contains $\rho$ (up to $\GL_2(\Om)$-isomorphism), and a subspace $N_q
\subseteq \coho^1(G_q, \Ad)$ such that:

\begin{itemize}
 \item All the elements of $C_q$ are isomorphic over $GL_2(\Om)$ when restricted to inertia. 
 \item $N_q \subseteq \coho^1(G_q, \Ad)$ has codimension equal to
   $\dim \coho^2(G_q, \Ad)$.
 \item Every $u\in N_q$ preserves the mod $\pi^m$ reductions of elements of $C_q$ for $m \geq \alpha$.
\end{itemize}

In other words, there exists a
smooth deformation condition containing $\rho$ of dimension equal to
$\dim \coho^1(G_q, \Ad) -\dim \coho^2(G_q, \Ad)$ if we start
lifting from a big enough exponent.
\end{prop}

In order to prove the proposition, let us recall the main results about types of deformations and types of reduction mod $\pi$.
Although they are mainly known, the results and its proofs are contained in section $2$ of \cite{arxiv}.

\begin{prop}
Let $q \neq 2$, be a prime number, with $q \neq p$. Then every
representation $\rho:G_q \rightarrow \GL_2(\F)$, up to twist by a
character of finite order, belongs to one of the following three
types:
\begin{itemize}
\item {\bf Principal Series:}
$\rho \simeq \left(\begin{smallmatrix}
\phi&0\\
0&1\\ 
\end{smallmatrix}\right) \hspace{0.25cm} \text{or} \hspace{0.25cm}
\rho \simeq \left(\begin{smallmatrix}
1&\psi\\
0&1\\
\end{smallmatrix}\right).$

\item {\bf Steinberg:} $\rho \simeq \left(\begin{smallmatrix}
\chi&\mu\\
0&1\\ 
\end{smallmatrix}\right),$ where $\mu \in \coho^1(G_q,\F(\chi))$ and 
$\mu|_{I_q} \neq 0$.
 
\item {\bf Induced:} $\rho \simeq \Ind_{G_M}^{G_q}(\xi)$, where
  $M/\Q_q$ is a quadratic extension and $\xi: G_M \rightarrow
  \F^\times$ is a character not equal to its conjugate under the
  action of $\Gal(M/\Q_q)$.
\end{itemize}
Here $\phi:G_q \to \F^\times$ is a multiplicative character and
$\psi:G_q \to \F$ is an unramified additive character.
\label{prop:Diamond}
\end{prop}

\begin{prop}
  Let $\tilde{\rho}:G_q \rightarrow \GL_2(\overline{\Z_p})$ be a
  continuous representation. Then up to twist (by a finite order
  character times and unramified character) and
  $\GL_2(\overline{\Z_p})$ equivalence we have:
  \begin{itemize}
  \item {\bf Principal Series:} $\tilde{\rho} \simeq
    \left(\begin{smallmatrix}
        \phi & \pi^n(\phi-1)\\
        0 & 1\end{smallmatrix} \right)$, with $n \in \Z_{\leq 0}$ satisfying
    $\pi^{n}(\phi-1) \in \overline{\Z_p}$ or $\tilde{\rho} \simeq
    \left(\begin{smallmatrix}
        1&\psi\\
        0&1\\
\end{smallmatrix}\right)$.

\item {\bf Steinberg:} $\tilde{\rho} \simeq \left(\begin{smallmatrix}
      \chi & \pi^n\mu \\
      0 & 1\end{smallmatrix} \right)$, with $n \in \Z_{\geq 0}$.

\item {\bf Induced:} There exists a quadratic extension $M/\Q_q$
  and a character $\xi: G_M \rightarrow \overline{\Z_p}^\times$ not
  equal to its conjugate under the action of $\Gal(M/\Q_q)$ such
  that $\tilde{\rho} \simeq \<v_1,v_2>_{\O_L}$, where for $\alpha$ a
  generator of $\Gal(M/\Q_p)$ and $\beta \in G_M$, the action is given by

\[
  \beta(v_1)= \xi(\beta)v_1,\quad \beta(v_2)= \xi^{\alpha}(\beta)v_2, \quad  \alpha(v_1)= v_2 \quad \text{and} \quad
  \alpha(v_2)=\xi(\alpha^2)v_1,
\]
%
or
\[
  \tilde{\rho}(\beta)=\begin{pmatrix}
               \xi(\beta)&\frac{\xi(\beta)-\xi^{\alpha}(\beta)}{\pi^{n}}\\
               0&\xi^{\alpha}(\beta)
              \end{pmatrix}   \qquad \text{ and } \qquad \tilde{\rho}(\alpha)=\begin{pmatrix}
               -a&\frac{\xi(\alpha^2)-a^2}{\pi^n}\\
               \pi^n&a
              \end{pmatrix}  
\]
where $\xi^\alpha$ is the character of $G_M$ defined by $\xi^\alpha(g) = \xi(\alpha g\alpha^{-1})$ and $a \in \Om_L^\times$.
Observe that when $M/\Q$ is ramified we can take $\alpha$ and $\beta$ to be a Frobenius element and a generator of the tame inertia respectively.
\end{itemize}
\label{types}
\end{prop}

\begin{prop}
Let $\trho$ be as above and $\rrho$ its mod $p$ reduction. We have the following types of reduction:
\begin{itemize}
\item If $\trho$ is Principal Series, then $\rrho$ can be Principal Series or Steinberg, and the latter occurs only when $q \equiv 1 \pmod p$.
\item If $\trho$ is Steinberg, then $\rrho$ can be Steinberg or Principal Series, and the latter occurs only when $\rrho$ is unramified.
\item If $\trho$ is Induced, then $\rrho$ can be Induced, Steinberg or unramified Principal Series. For the last two cases we must have
  $q \equiv -1 \pmod p$ and $M/\Q$ ramified.
\end{itemize}
\label{prop:localtypesreductions}
\end{prop}

To prove Proposition~\ref{localtypes} we consider all the possible
pairs of $\GL_2(\F)$ and $\GL_2(\Om)$-equivalence classes for $\rho$ and $\rrho$ (indexing
them first by the class of $\rrho$), and for each of them
we define the corrsponding deformation class and cohomology subspace.\\

\noindent $\bullet$ {\bf Case 1: $\rrho$ is ramified Principal
Series.}  Proposition \ref{prop:localtypesreductions} implies that a
mod $\pi$ Principal Series can only come from a characteristic $0$
principal series.  The full study of this case is done in Case 1 of
Section 4 of [CP14]. The work there is done for $\Om$ unramified but
the same applies in our situation.\\

\noindent $\bullet$ {\bf Case 2: $\rrho$ is Steinberg.}  When $\rrho$
is Steinberg, it can be the reduction of any of the three
characteristic $0$ ramified types:

$\cdotp$ {\bf Case 2.1: $\rho$ is Steinberg}. The definition of $C_q$
and $N_q$ is essentially the same as \cite{arxiv}, Case 2 of Section
4. Although in that work only the case where $\Om/\Z_p$ is unramified
is treated, the ramification of $\Om$ does not affect the results.

$\cdotp$ {\bf Case 2.2: $\rho$ is Principal Series}. Proposition
\ref{prop:localtypesreductions} implies that $q \equiv 1 \pmod{p}$.
Let $\rho = \begin{pmatrix}
  \psi&*\\
  0&1
\end{pmatrix}$. Without loss of generality, we can take $*(\tau)=1$. We have the following lemma:

\begin{lemma}
 A deformation $\tilde{\rho}: G_q \rightarrow \GL_2(\Om/\pi^m)$ which has the form:

\[
\tilde{\rho}(\tau) = \begin{pmatrix}
\psi(\tau)&1\\
0&1
\end{pmatrix} \hspace{0.2cm} \text{and} \hspace{0.2cm} \tilde{\rho}(\sigma) = \begin{pmatrix}
\alpha&\gamma\\
0&\beta
\end{pmatrix},
\]
has a unique lift to characteristic zero of the same form if and only
if $\beta^2 + \gamma(\psi(\tau)-1)\beta + \Psi \equiv 0 \pmod{\pi^m}$,
where $\Psi$ is a fixed lift determinant.
\label{lemma:Khare}
\end{lemma}
\begin{proof}
  This Lemma is part of a computation made in Proposition 3.4 of
  \cite{khare}. There, it is done for deformations with coefficients in unramified coefficient field
  and its mod $p$ reductions, but the same proof works in general. 
\end{proof}

Let $j \in \coho^1(G_q, \Ad)$ be the element defined by:
\[
j(\sigma) = e_2 \hspace{0.2cm} \text{,} \hspace {0.2cm}j(\tau) = 0,\] 
and take $N_q$ to be the subspace it generates. Also let $C_q$ be the
set of deformations to $\Om$ which have the form given in
Lemma~\ref{lemma:Khare}. Observe that any mod $\pi^m$ reduction of an
element of $C_q$ satisfies the equation $\beta^2 +
\gamma(\psi(\tau)-1)\beta + \Psi = 0$, and acting by $j$ on it does
not affect this (as $\pi$ divides $\psi(\tau)-1$ so adding a multiple
of $\pi^{m-1}$ to $\gamma$ does not change the equation modulo
$\pi^m$). Then, Lemma \ref{lemma:Khare} guarantees that $C_q$ and
$N_q$ satisfy the property we are looking for.\\

$\cdotp$ {\bf Case 2.3: $\rho$ is Induced}. Proposition
\ref{prop:localtypesreductions} tells us that necessarily $q \equiv -1
\pmod{\pi}$ and
by results in Section $3$ of [CP14] we have that $d_1 = d_2 = 1$. Therefore we can take $N_q = \set{0}$ and $C_q = \set{\rho}$.\\

\noindent $\bullet$ {\bf Case 3: $\rrho$ is Induced.}  By Proposition
\ref{prop:localtypesreductions}, when $\rrho$ is Induced, $\rho$ must
be Induced as well. The choice of $C_q$ and $N_q$ in this case is
explained in Case $3$ of Section $4$ of [CP14].\\

\noindent $\bullet$ {\bf Case 4: $\rrho$ is unramified.}
Being unramified, $\rrho$ allows lifts to any type of deformation. We must treat each of them separately as they have many subcases.
The case of $\rho$ being {\bf Steinberg} is dealt with in Case 4 of Section 4 of [CP14]. There are two other cases left to study.
In each of them we will distinguish between three types of equivalence classes for $\rrho$, acording to the image of Frobenius.
This case is the one that includes bad primes, the calculations made here show where the badness condition appears.\\

{\bf Case 4.1: $\rho$ is Principal Series.} In this case we have
$\rho = \left(\begin{smallmatrix}
    \phi & \pi^{-r}(\phi-1)\\
    0 & 1\end{smallmatrix} \right)$ with $r \geq 0$. By Proposition
\ref{prop:localtypesreductions} we necessarily have
$q \equiv 1 \pmod{\pi}$.
        
$\cdotp$ If $\rrho(\sigma) = \left(\begin{smallmatrix}
\beta & 0\\
0 & 1\end{smallmatrix} \right)$ with $\beta \neq 1$ we have $d_1 = 2$ and $d_2 = 1$ so we are looking for a one-dimensional
subspace $N_q$. Observe that necessarily $\rho \simeq \left(\begin{smallmatrix}
\phi & 0\\
0 & 1\end{smallmatrix} \right)$ over $\Om$ as $\psi(\sigma) \equiv \beta \nequiv 1 \pmod{\pi}$ so $(\psi - 1)/\pi$ 
is not an integer.

Let $u \in \coho^1(G_q,\Ad)$ be the cocycle defined by $u(\sigma) = e_1$ and $u(\tau) = 0$. We can take $N_q = \langle u \rangle$ and $C_q$ the
set of representations of the form $\rho \simeq \left(\begin{smallmatrix}
\phi\psi & 0\\
0 & \psi^{-1}\end{smallmatrix} \right)$ for $\psi: G_q \to \Om^\times$ unramified. It is easily checked that these satisfy the desired properties.\\

$\cdotp$ If $\rrho(\sigma_q) = \left(\begin{smallmatrix}
1 & 1\\
0 & 1\end{smallmatrix} \right)$ the corresponding dimensions are $d_1 = 2$ and $d_2 = 1$. In this case we have $\rho \simeq \left(\begin{smallmatrix}
\phi & \frac{\phi -1}{\pi^r}\\
0 & 1\end{smallmatrix} \right)$,
with $v_\pi(\phi -1) = r$. Now take the cocycle defined by $u(\sigma) = 0$ and $u(\tau) = e_2$. 
We take $\alpha = v(x-y) + 2$ and set $N_q = \langle u \rangle$ and $C_q$ the set of deformations of $\rrho$ such that
$\rho \simeq \left(\begin{smallmatrix}
\gamma\phi & \beta\frac{\gamma\phi - \gamma^{-1}}{\pi^r}\\
0 & \gamma^{-1}\end{smallmatrix} \right)$
for some $\beta \in \Om^\times$ and $\gamma$ an unramified character congruent to $1$ modulo $\pi^\alpha$. 

\begin{lemma}
\label{jordan}
 The set $C_q$ and subspace $N_q$ defined above satisfy that $N_q$ preserves $C_q$ mod $\pi^m$ for all $m$ such that $\phi$ is ramified mod $\pi^m$.
\end{lemma}

\begin{proof}
Assume that $\rho(\sigma) = \left(\begin{smallmatrix}
                             a&\frac{a-b}{\pi^r}\\
                             0&b
                            \end{smallmatrix}\right)$ and $\rho(\tau) = \left(\begin{smallmatrix}
                                                                        x&\frac{x-y}{\pi^r}\\
                                                                        0&y\end{smallmatrix}\right)$.
Our assumptions on $\rho$ not being bad tell us that $v(a-b) \geq v((x-y)/\pi^r)$. If we have
$v(a-b) > v(x-y)$ then we change $\sigma$ by $\tau\sigma$ (which is another Frobenius element). In this way, we can assume that 
$v(a-b) \leq v(x-y)$. 

We want to prove that if we have a deformation $\trho:G_q \to \GL_2(\Om/\pi^{m})$ that sends
$$ \trho(\sigma) = \begin{pmatrix}
                             a&c\\
                             0&b
                            \end{pmatrix}\hspace{0.2cm} \text{and} \hspace{0.2cm} \trho(\tau) = \begin{pmatrix}
                                                                        x&z\\
                                                                        0&y\end{pmatrix},$$
with $x \neq y$, and is the reduction of some element of $C_q$ (i.e. $c = \beta(a-b)$ and $z = \beta(x-y)$ for some $\beta,a,b,c,x,y,z \in \Om$), 
then $u\cdot\trho$ is also the reduction of some element of $C_q$. Recall that 
$$ u\cdot\trho(\sigma) = \begin{pmatrix}
                             a&c\\
                             0&b
                            \end{pmatrix}\hspace{0.2cm} \text{and} \hspace{0.2cm} u\cdot\trho(\tau) = \begin{pmatrix}
                                                                        x&z+\pi^{m-1}\\
                                                                        0&y\end{pmatrix}.$$
                                                                        
It is easily checked that the cocycle $v$ that sends $\sigma$ to $\lambda_1e_1 + \lambda_2e_2$ and $\tau$ to $0$ is a coboundary
for any choice of $\lambda_1, \lambda_2 \in \F$. So $\trho$ can also be tought as
\[
 u\cdot\trho(\sigma) = \begin{pmatrix}
                             a(1+\lambda_1\pi^{m-1})&c+\lambda_2\pi^{m-1}\\
                             0&b(1+\lambda_1\pi^{m-1})^{-1}
                            \end{pmatrix}\hspace{0.2cm} \text{and} \hspace{0.2cm} u\cdot\trho(\tau) = \begin{pmatrix}
                              x&z+\pi^{n-1}\\
                              0&y\end{pmatrix}
\]
for any $\lambda_1, \lambda_2 \in \F$. Therefore, it is enough to find some $\lambda_1, \lambda_2 \in \Om$
such that 
\[
\frac{x-y}{z+\pi^{m-1}} = \frac{a(1+\lambda_1\pi^{m-1})-b(1+\lambda_1\pi^{m-1})^{-1}}{c+\lambda_2\pi^{m-1}}.
\]
given that 
\[
\frac{x-y}{z} = \frac{a-b}{c}.
\]
Expanding this equation we find out that it is equivalent to 
\[
\lambda_1(z+\pi^{m-1})(a+b(1+\lambda_1\pi^{m-1})^{-1}) + a-b = \lambda_2(x-y). 
\]
which has a solution given that $v(a-b) \geq v(z)$ (we solve first for $\lambda_1$ in order for both sides to have the same valuation, and then
there is a $\lambda_2$ that makes the equality true).
\end{proof}

\begin{remark}
Observe that when the condition $v(a-b) \geq v(z)$ does not hold, the last equation of the proof does not have a solution, since 
no matter which $\lambda_1, \lambda_2 \in \Om$ we pick, the valuation of the left hand side is $v(a-b)$ and the valuation
of the right hand side is bigger or equal than $v(x-y) \geq v(z) > v(a-b)$. Moreover, following the same type of computations
we can prove that for the chosen set $C_q$ there is no non-trivial $u \in \coho^1(G_q,\Ad)$ such that the subspace $\langle u \rangle$
preserves $C_q$.

\end{remark}

$\cdotp$ If $\rrho(\sigma_q) = \left(\begin{smallmatrix}
    1 & 0\\
    0 & 1\end{smallmatrix} \right)$ we have $d_1 = 6$ and $d_2 = 3$
therefore we need to find a subspace $N_q$ of dimension $3$.  This
case is a little more involved than the other two as there are
non-trivial elements of $\coho^1(G_q,\Ad)$ that act trivially on modulo
$\pi^m$ deformations for high powers of $\pi$. We follow the same ideas as in the
study of the Steinberg-reducing-to-unramified case
(the spirit of these ideas is taken from the approach to trivial primes followed in \cite{ravihamblen}).

Assume first that $\rho \simeq \left(\begin{smallmatrix}
    \phi & 0\\
    0 & 1\end{smallmatrix} \right)$. We will take $C_q$ as the set of representations of the form 
\[
\rho' \simeq \begin{pmatrix}
  \gamma\phi&0\\
  0&\gamma^{-1}
                                                                                                        \end{pmatrix}
\]
with $\gamma:G_q \to \Om^\times$ an unramified character, that lift the reduction mod $\pi^\alpha$ of $\rho$, 
with $\alpha = v(\phi_1(\tau) - \phi_2(\tau)) + 2$. Clearly, the
set $C_q$ is preserved by the cocycle $u_1$ that sends $\sigma \mapsto
e_1$ and $\tau \mapsto 0$.

We will construct two more cocycles $u_2$ and $u_3$ that act trivially
on reductions modulo $\pi^m$ of deformations on $C_q$ for $m\geq \alpha$. 
Let $\trho: G_q \to \GL_2(\Om/\pi^m)$ be the mod $\pi^m$
reduction of an element in $C_q$. Let $\trho(\sigma) =
\left(\begin{smallmatrix}
    a&0\\
    0&b
\end{smallmatrix}\right)$ and $ \trho(\tau) = \left(\begin{smallmatrix}
  x&0\\
  0&y \end{smallmatrix}\right)$.  Let $u \in \coho^1(G_q,\Ad)$. To
prove that $u$ acts trivially on $\trho$ we need to find a matrix $C
\in \GL_2(\Om/\pi^m)$ such that $C \equiv Id \pmod{\pi}$ and $C\trho
C^{-1} = (Id + \pi^{m-1}u)\trho$. One can find such matrix by taking $C =
\left(\begin{smallmatrix}
    1+\pi\alpha&\pi\beta\\
    \pi\gamma&1+\pi\delta
\end{smallmatrix}\right)$ 
and explicitly computing $C \trho = (Id + \pi^{m-1}u)\trho C$ at $\sigma$ and $\tau$.  
In this way, one finds out that if $v(x-y) > v(a-b)$ 
then the cocycles $u_2$ and $u_3$ sending $\sigma$ to $e_2$
and $e_3$ respectively and $\tau$ to $0$ act trivially on the reductions of elements of $C_q$. 
The corresponding base change matrices $C_i$ that conjugate $(Id+\pi^{m-1}u_i)\trho$ into $\trho$ are
\[
C_2 = \begin{pmatrix}
         1&\frac{\pi^{m-1}}{a-b}\\
         0&1
        \end{pmatrix} \hspace{0.2cm} \text{and} \hspace{0.2cm} C_3 = \begin{pmatrix}
         1&0\\
         -\frac{\pi^{m-1}}{a-b}&1
        \end{pmatrix}.
\]
It remains to check what happens when $v(\phi_1(\tau) - \phi_2(\tau)) \leq v(\phi_1(\sigma) - \phi_2(\sigma))$. 
Observe that we can always assume that $v(\phi_1(\tau) - \phi_2(\tau)) = v(\phi_1(\tau) - \phi_2(\tau))$ in this case, by simply 
changing $\sigma$ for $\tau\sigma$. 
In this case we can take $u_2$ sending $\sigma$ to $e_2$ and $\tau$ to 
$\lambda e_2$ and $u_3$ sending $\sigma$ to $e_3$ and $\tau$ to $\lambda e_3$ for $\lambda = (x-y)/(a-b) \in \F$ 
(notice that this does not depend on $\trho$). 
Again, the action of both cocycles will be trivial and the base change
matrices will be the same as before.\\

It remains to consider the case where $\rho \simeq
\left(\begin{smallmatrix}
    \phi & (\phi-1)/\pi^r\\
    0 & 1\end{smallmatrix} \right)$ for $r> 0$. Let $\alpha = v((\phi-1)/\pi^r) + 2$ and let $C_q$ be the set
of deformations of the mod $\pi^\alpha$ reduction of $\rho$ such that $\rho \simeq
\left(\begin{smallmatrix}
    \gamma\phi & \beta\frac{\gamma\phi -\gamma^{-1}}{\pi^r}\\
    0 & \gamma^{-1}\end{smallmatrix} \right)$ for some $\beta \in
\Om^\times$ and $\gamma$ an unramified character congruent to $1$ modulo $\pi^\alpha$. By doing the exact same calculation as in Lemma
\ref{jordan} it can be proved that the cocycle $u_1$ that sends
$\sigma$ to $0$ and $\tau$ to $e_2$ preserves the reductions of
elements in $C_q$, given that $\rho$ is not bad. We still need two more elements preserving
$C_q$. As in the previous case, we have two cocycles that act
trivially on mod $\pi^m$ reductions of elements of $C_q$ for $m\geq \alpha$. 
Let $\trho$ is a mod $\pi^m$ reduction of some element in
$C_q$ given by $\trho(\sigma) = \left(\begin{smallmatrix}
    a&c\\
    0&b \end{smallmatrix}\right)$ and $\trho(\tau) =
\left(\begin{smallmatrix}
    x&z\\
    0&y\end{smallmatrix}\right)$. 
    
    If $v(a-b) < v(x-y)$, we take $u_2$ sending $\sigma$ to
$e_1$ and $\tau$ to 0 and $u_3$ the one that sends
$\sigma$ to $e_2$ and $\tau$ to $0$. We claim that these act
trivially on $\trho$ if $m\geq v(z)+2$.

    If otherwise $v(a-b) \geq v(x-y)$ we can assume that $v(a-b) = v(x-y)$ as in Lemma \ref{jordan}. 
    Let $$\lambda = \frac{x-y}{a-b} =
\frac{z}{c}.$$ Let $u_2$ be the cocycle that sends $\sigma$ to
$e_1$ and $\tau$ to $\lambda e_1$ and $u_3$ the one that sends
$\sigma$ to $e_2$ and $\tau$ to $\lambda e_2$. These act trivially on $\trho$ if $m\geq v(z)+2$.

It can be checked that in both cases the base change matrices given by 
$$C_2 = \begin{pmatrix}
         1&0\\ \frac{-\pi^{m-1}}{c}&1
       \end{pmatrix} \hspace{0.2cm} \text{and} \hspace{0.2cm} C_3
       = \begin{pmatrix}1+\frac{\pi^{m-1}}{c}&0\\0&1\end{pmatrix}$$
       serve to prove the trivialness of the action of $u_2$ and $u_3$ respectively. This concludes the case.
       
\medskip

{\bf Case 4.2: $\rho$ is Induced.}  Proposition
\ref{prop:localtypesreductions} says that whenever $\rho$ is induced
and $\rrho$ is unramified, $\rrho(\sigma)$ has eigenvalues $1$ and
$-1$ (up to twist) and $q \equiv -1 \pmod{p}$. In this case we have
that $d_1=3$ and $d_2=2$.  We want to find a set $C_q$ and a subspace
$N_q$ of dimension $1$ preserving it. Let $C_q = \set{\rho}$.  As in
the Principal Series case, we will be able to find non trivial
cocycles that act trivially on mod $\pi^m$ reductions of $\rho$ for
$m$ big enough. We split into the two possible families of
$\GL_2(\overline{\Z_p})$-equivalence classes of induced representations given by
Proposition \ref{types}. 

\medskip

$\cdotp$ If $\rho(\sigma) = \left(\begin{smallmatrix}
    0&t\\
    1&0
  \end{smallmatrix}\right)$ and $\rho(\tau) =
\left(\begin{smallmatrix}x&0\\ 0&y\end{smallmatrix}\right)$, it can
be checked that the cocycle $u$ sending $\sigma$ to $0$ and $\tau$ to
$e_2 - e_3$ is non trivial. 
The cocycle $u$ acts trivially modulo $\pi^m$ for all $m \geq v(x-y)+2$ whenever $v(t-1) > v(x-y)$.
 In this case the base change matrix is given by
$$C = \begin{pmatrix}
       1&-\frac{\pi^{m-1}}{x-y}\\
       -\frac{\pi^{m-1}}{x-y}&1\end{pmatrix}.$$
       
If we are in a case where $v(t-1) \leq v(x-y)$ then we can go back to a case where $v(t-1) > v(x-y)$ by twisting $\rho$.
Let $\eta \in 1 + \pi\O$ be an unit such that $v(\eta^2t - 1) > v(x-y)$ and $\gamma:G_q \to \O^\times$ an unramified character
mapping $\sigma$ to $\eta$. If we twist $\rho$ by $\gamma$ then the deformation obtained is equivalent to $\rho'$ sending

$$\rho'(\sigma) = \begin{pmatrix}
                   0&\eta^2t\\
                   1&0
                  \end{pmatrix} \hspace{0.5cm} \text{and} \hspace{0.5cm} \rho'(\tau) = \begin{pmatrix}
                   x&0\\
                   0&y
                  \end{pmatrix} 
$$
 via the base change matrix $$C = \begin{pmatrix}
                                  \eta&0\\0&1
                                 \end{pmatrix}.$$

$\cdotp$ If $\rho(\sigma) = \left(\begin{smallmatrix}
                 -a&c\\
                 b&a
               \end{smallmatrix}\right)$ 
and $\rho(\tau) =
             \left(\begin{smallmatrix} x&z\\
                 0&y\end{smallmatrix}\right),$
as in the previous case, it can be checked that the cocycle $u$ that
sends $\sigma$ to $0$ and $\tau$ to $e_2$ is non trivial. Again, the
action of this cocycle in the reduction modulo $\pi^m$ of $\rho$ is
trivial for $m \geq v(z)+2$. The base change matrix that works for this case is
\[
C = \begin{pmatrix}
       1+\frac{\pi^{m-1}}{z}&\frac{b\pi^{m-1}}{2az}\\
       \frac{c\pi^{m-1}}{2az}&1\end{pmatrix}.
\]

\section{Local deformation theory at $p$}

At the prime $p$ we will impose the deformation condition of being
``nearly ordinary'' (as in \cite{CM}). This section is mainly about
gathering previously done calculations, and all the deformations
appearing are deformations of the local Galois group $G_p$.

\begin{defi}
 We say that a deformation of $G_p$ is ``nearly ordinary'' if its restriction to the inertia subgruop is upper-triangular and its 
 semisimplification is not scalar, i.e. if
 $$\rho|_{I_p} = \begin{pmatrix}
                  \psi_1&*\\
                  0&\psi_2
                 \end{pmatrix}.
 $$ with $\psi_1 \neq \psi_2$.
\end{defi}

We will prove the following theorem.

\begin{thm}
 \label{localatp}
 Let $\rho_n:G_p \to \Om/\pi^n$ be a nearly ordinary deformation and $\rrho$ its mod $\pi$ reduction. There is
 a family of nearly ordinary deformations $C_p$ to characteristic $0$ such that $\rho_n$ is the reduction of a member of $C_p$
 and a subspace $N_p \subseteq \coho^1(G_p,\Ad)$ of codimension equal to $\dim \coho^2(G_p, \Ad)+1$ preserving $C_p$ in the
 sense of Proposition \ref{localtypes}. 
\end{thm}

\begin{proof}
 Let $$\rrho \simeq \begin{pmatrix}
                         \psi_1&*\\
                         0&\psi_2
                        \end{pmatrix}.$$
 By twisting $\rrho$ by $\psi_2^{-1}$ we can assume that $\psi_2 = 1$. Let $U$ be the set of upper-triangular $2\times2$ matrices 
 of trace $0$. To prove the theorem we will construct for each possible $\rrho$, the corresponding set and subspace and verify that
 its dimensions satisfy the statement of the theorem. In most of the cases $C_p$ will consist on all nearly ordinary deformations
 of $\rrho$ and $N_p$ will be the image of $\coho^1(G_p,U)$ in $\coho^1(G_p,\Ad)$. 
 
 In order to see this, we simply compute the dimension of the image of $\coho^1(G_p,U)$ in $\coho^1(G_p,\Ad)$ and compare it with
 $\dim \coho^1(G_p,\Ad) - \dim \coho^2(G_p,\Ad) - 1$. We can reduce the computation of all these values to finding the dimensions of
 $\coho^0(G_p,U)$, $\coho^0(G_p,U^*)$, $\coho^0(G_p,\Ad)$ and $\coho^0(G_p,(\Ad)^*)$ by using local Tate duality and the formula for 
 the Euler-Poincare characteristic (which is equal to $2$ for $U$ and $3$ for $\Ad$). 
 The kernel of the map $\coho^1(G_p,U) \to \coho^1(G_p,\Ad)$ induced by $U \subseteq \Ad$ is contained in $\coho^0(G_p,\Ad/U)$, which 
 is equal to $0$ given that $\psi_1 \neq \psi_2$. With all these tools, the required dimensions are easily computed and we obtain 
 that taking $C_p$ as the set of all nearly ordinary deformations and $N_p$ as the image of $\coho^1(G_p,U)$ in $\coho^1(G_p,\Ad)$ works
 for all cases but the one in which $\rrho$ is decomposable and $\psi_1$ is the cyclotomic character (recall we are assuming $\psi_2 = 1$).
 

 In this case the universal ring for nearly ordinary deformations is not smooth and $C_p$ is not preserved by $N_p$ as there are some mod $\pi^s$
 nearly ordinary deformations of $\rrho$ that do not lift back to characteristic $0$. We need to take a smaller set $C_p$ in this case. 
 In order to solve this, we claim that the universal deformation ring for nearly ordinary lifts of $\rrho$ with fixed determinant $\psi$ is 
 isomorphic to the universal deformation ring for \textbf{ordinary} lifts of $\rrho$ with arbitrary determinant. To see this, just observe
 that from any nearly ordinary deformation of $\rrho$ to a ring $A$ we can obtain an ordinary lift of $\rrho$ by twisting by inverse of the 
 character appearing in the place $(2,2)$. To go the other way round, if we have an ordinary deformation $\tilde{\rho}$ of $\rrho$ to $A$
 given by $$ \tilde{\rho} = \begin{pmatrix}
                             \omega_1&*\\0&\omega_2
                            \end{pmatrix}$$
where $\omega_2$ is unramified and want to obtain a nearly ordinary deformation of $\rrho$ with determinant $\psi$, we need to twist by a square root of
$\psi(\omega_1\omega_2)^{-1}$. This character has a square root by Hensel's lemma, as its reduction modulo the maximal ideal is $1$ which has a
square root. Given this identification, the wanted result follows from Proposition $6$ of \cite{KR}, where the same result is proven for the
ordinary arbitrary determinant case.
 \end{proof}


\section{Global deformation theory}

In this section we prove the one of the main results of this article.

\begin{thm} 
 \label{abstract}
 Let $n \geq 2$ be an integer and $\rho_n: G_\Q \to \GL_2(\Om/\pi^n)$
 be a continuous representation which is odd and nearly
 ordinary at $p$.  Assume that
 $\Img(\rho_n)$ contains $\SL_2(\Om/p)$ if $n \ge e$ and $\SL_2(\Om/\pi^n)$ otherwise. 
 Let $P$ be a set of primes of $\Q$ containing
 the ramification set of $\rho_n$.  For each $v \in P \backslash
 \set{p}$ fix a local deformation $\rho_v: G_v \to \GL_2(\Om)$ that
 lifts $\rho_n|_{G_v}$ and is not bad (see Definition~\ref{bad}).
 Then there exists a continuous representation $\rho: G_\Q \to
 \GL_2(\Om)$ and a finite set of primes $R$ such that:
 \begin{itemize}
  \item $\rho$ lifts $\rho_n$, i.e. $\rho \equiv \rho_n \pmod{\pi^n}$.
  \item $\rho$ is unramified outside $P \cup R$.
  \item For every $v\in P$, $\rho|_{I_v} \simeq \rho_v|_{I_v}$ over $\GL_2(\Om)$.
  \item $\rho$ is nearly ordinary at $p$.
  \item All the primes of $R$, except possibly one, are not congruent to $1$ modulo $p$.
 \end{itemize}
\end{thm}

The proof of this theorem essentially consists on finding a way to
lift $\rho_n$ to characteristic $0$ one power of $\pi$ at a time.
We will split the proof into two sections, essentially because the
local results we have so far are split in two different cases
depending on whether the exponent $m$ in each step is big enough or not.  
Let $\alpha$ be the integer obtained in the following way:
for each $v \in P$, Proposition \ref{localtypes} gives an integer $\alpha_v$ such that
there is a set $C_v$ containing $\rho_v$ and a subspace $N_v$ preserving its reductions mod $\pi^n$ for $n \geq \alpha_v$.
Let $\alpha$ be the maximum of the $\alpha_v$'s for $v\in P$.
When lifting from $\pi^m$ to $\pi^{m+1}$ for $m \geq \alpha$ we are in a global setting similar
to the one in \cite{ravi3}. The existence of the sets $C_v$ and
subspaces $N_v$ let us mimic the argument given there. When working
modulo $m$ for $m < \alpha$, we do not count with these sets and subspaces for all the primes of $P$  
and therefore are unable to overcome the local
obstructions in the same fashion as before. In this case, we will follow the ideas from
\cite{ravilarsen} and will lift to $\Om/\pi^\alpha$ by adding a finite number of auxilliary primes at each power of $\pi$.

\medskip 

As $\Om$ is the ring of integers of a ramified extension we have that
$\Om/\pi^2$ is isomorphic to the dual numbers and therefore the
projection mod $\pi^2$ of $\rho_n$ defines an element in
$\coho^1(G_\Q, \Ad)$ which we will call $f$. Observe that our hypotheses imply that the 
image of the projection of $\rho_n$ to $\Om/\pi^2$ contains $\SL_2(\Om/\pi^2)$ 
and thus $f \neq 0$ as an element in $\coho^1(G_\Q,\Ad)$.

\subsection{Getting to mod $\pi^{\alpha}$}

Assume that the exponent $n$ we start with is strictly smaller than
the natural number $\alpha$ from Proposition \ref{localtypes} (if this
is not the case we are done). The idea is to adjust the main argument
of \cite{ravilarsen} to our situation. Recall the following
definition.

\begin{defi}
  A prime $q$ is \emph{nice} for $\rrho$ if it satisfies the following properties
 \begin{itemize}
 \item The prime $q$ is not congruent to $\pm 1$ mod $p$.
 \item The representation $\rho_n$ is unramified at $q$.
 \item The eigenvalues of $\rrho(\sigma)$ have ratio $q$.
 \end{itemize}
 We say that $q$ is nice for $\rho_n$ if furthermore
 \begin{itemize}
  \item The eigenvalues of $\rho_n(\sigma)$ have ratio $q$.
 \end{itemize}
 
 At a nice prime $q$ we consider the set
\[
C_q= \{\text{deformations } \rho \text{ of }\rrho|_{G_q} \; : \;
\rho(\sigma) = \left(\begin{smallmatrix} q&0\\0&1
 \end{smallmatrix}\right)\},
\]
and the subspace $N_q \subseteq \coho^1(G_q,\Ad)$ generated by the
cocycle $u$ sending $\sigma$ to 0 and $\tau$ to
$e_2$. It is
easy to check that $u$ preserves the reductions of elements of $C_q$.
  
\end{defi}

The work on \cite{ravilarsen} is based on the existence of nice primes
that are either zero or non zero at certain elements of both
$\coho^1(G_U,\Ad)$ and $\coho^1(G_U,(\Ad)^*)$ for different sets of
primes $U$. We claim that the same arguments work in this settings
except for the cases when the element $f \in \coho^1(G_P,\Ad)$
attached to $\rho_n$ mod $\pi^2$ is involved. We will sort this obstacle by adding
the following primes.

\begin{defi}
 We say that a prime $q$ is a \textit{special prime} for $f$ if 
 \begin{itemize}
 \item The representation $\rho_n$ is unramified at $q$.
 \item $\rho_n(\sigma) = \left(\begin{smallmatrix}
1 & \pi \\
0 & 1
\end{smallmatrix}\right)$.
\item The prime $q \equiv 1 \pmod{\pi^n}$. 
 \end{itemize}
\end{defi}
Note that for such primes $q$ we have that $f|_{G_q} \neq 0$ as an element in $\coho^1(G_q,\Ad)$. We need some partial results to state and prove the main result of this case (Theorem~\ref{lifta}).

\begin{lemma}
\label{nice}
 Let $f,f_1,\ldots, f_r \in \coho^1(G_\Q,\Ad)$ and $\phi_1, \ldots, \phi_s \in \coho^1(G_\Q, (\Ad)^*)$ be linearly independent sets.
 \begin{enumerate}[a)]
  \item Let $I\subseteq \set{1,\ldots r}$ and $J \subseteq \set{1,\ldots,s}$. There is a Chebotarev set of primes $v$ such that
  \begin{itemize}
   \item $v$ is nice for $\rho_n$.
   \item $f_i|_{G_v} \neq 0$ if $i\in I$ and $f_i|_{G_v} = 0$ if $i\notin I$.
   \item $\phi_j|_{G_v} \neq 0$ if $j\in J$ and $\phi_j|_{G_v} = 0$ if $j\in J$
  \end{itemize}
  \item Also, there is a Chebotarev set of primes $w$ such that
  \begin{itemize}
   \item $w$ is a special prime for $f$, henceforth $f|_{G_w} \neq 0$.
   \item $f_i|_{G_w} = 0$ for all $1 \leq i \leq r$.
  \end{itemize}
 \end{enumerate}
\end{lemma}

\begin{remark}
  For special primes we can also define a set $C_v$ of deformations to
  characteristic $0$ and a subspace $N_v \subseteq \coho^1(G_v,\Ad)$
  of codimension $\dim \coho^2(G_v,\Ad)$ preserving it. This is explicitly done in Lemma 4.1
  of \cite{arxiv}. The procedure is the same as the one employed in cases 4.1 and 4.2 of 
  Section $2$. We will not reproduce the results here in an effort to preserve elegance. Notice that
  this corresponds to the case of a trivial local $\rrho$ lifting to a Steinberg $\rho$.
\end{remark}

\begin{proof}
  This is a slight modification of Fact 5 of \cite{ravilarsen}. The
  main problem with nice primes in ramified extensions is that if $v$
  is a nice prime then $f|_{G_v} = 0$. The use of special primes for
  $f$ solves this problem, since almost by definition if $v$ is a
  special prime, $f|_{G_v} \neq 0$.  In order to check that Chebotarev
  conditions at the different $f_i$'s and $\phi_j$'s are disjoint from
  the condition of being nice for $\rho_n$, and that this last
  condition only overlaps with extension corresponding to the element
  $f$, we need to understand the Galois structure of the corresponding
  extensions. We will prove a ramified version of Lemma 5.8 of
  \cite{arxiv}.  Following the notation of that article (which is the
  original notation of \cite{ravi2}), let $K = \Q(\Ad)\Q(\mu_p)$.  For
  each $f_i$ let $L_i$ the extension of $K$ given by its kernel and
  for each $\phi_j$ let $M_j$ be the corresponding one.  Finally let $K' =
  K \cdot\Q(\Adn)$ and $L_f$ be the extension of $K$ given by $f$. Let $L =
  L_f\prod L_i$ and $M = \prod M_j$. We claim that $K' \cap LM = L_f$.
  
  For this, let $\HH  = \Gal(K'/K) \subseteq \PGL_2(\Om/\pi^n)$ and $\pi_1:
  \PGL_2(\Om/\pi^n) \to \PGL_2(\F)$.  Observe that $\HH$ consists on the
  classes of matrices in $\Img(\rho_n)$ which are trivial in
  $\PGL_2(\F)$, i.e. $\HH = \Img(\Adn) \cap \Ker(\pi_1)$.  Recall that
  our hypotheses imply $\PSL_2(\Om/\pi^n) \subseteq \Img(\Adn)
  \subseteq \PGL_2(\Om/\pi^n)$, and therefore $\PSL_2(\Om/\pi^n)\cap
  \Ker(\pi_1) \subseteq \HH \subseteq \Ker(\pi_1)$. As
  $[\PSL_2(\Om/\pi^n):\PGL_2(\Om/\pi^n)] = 2$ and $\Ker(\pi_1)$ is a
  $p$ group we have that $\HH = \Ker(\pi_1)$.

  Recall that $\Gal(F/K) \simeq (\Ad)^r \times (\Ad^*)^s$ as
  $\Z[G_\Q]$-module and by Lemma $7$ of \cite{ravi2}, this is
  its decomposition as $\Z[G_\Q]$ simple modules. This implies that
  $K' \cap LM$ is the direct sum of the quotients of $\Gal(K'/K) \simeq \HH$ isomorphic to $\Ad$ or $(\Ad)^*$.
  To prove that the only such quotient is $\Gal(L_f/K)$ observe that any 
  surjective morphism $\HH \to \Ad$ or $(\Ad)^*$ must contain $[\HH:\HH]$ inside its kernel. We will
  prove in Lemma \ref{conm2} below that such conmutator is equal to the subgroup of $\HH$ formed
  by the matrices congruent to the identity modulo $\pi^2$. This finishes the proof as implies that
  such quotient necessarily factors through $\Gal(L_f/K)$.

\end{proof}

\begin{lemma}
 \label{conm2}
If $H \subseteq \SL_2(\Om/\pi^n)$ is the subgroup consisting of matrices 
congruent to the identity modulo $p$ then its commutator subgroup $[H:H]$ is 
the subgroup $H'$ of $\SL_2(\Om/\pi^n)$ formed by the matrices congruent to the
identity modulo $\pi^2$. 
\end{lemma}

\begin{proof}
 It is easy to check that $H/H'$ is abelian, implying that $[H:H] \subseteq H'$.
 For the other inclusion, observe that $H'$ is generated by the set of elements 
 of the form
 
$$\begin{pmatrix}
  1&\pi^2x\\
  0&1
\end{pmatrix},\hspace{0,25cm} \begin{pmatrix}
  1&0\\
  \pi^2y&1
\end{pmatrix} 
\hspace{0,15cm} \text{and} \hspace{0,15cm}
\begin{pmatrix}
  1+\pi^2z&0\\
  0&(1+\pi^2z)^{-1}
\end{pmatrix} $$

for $x,y,z \in \Om$. It is easy to verify the following identities
for any $a,b \in \Om$:
\begin{itemize}
\item $\left[\left(\begin{smallmatrix}
  1&\pi a\\
  0&1
\end{smallmatrix}\right):\left(\begin{smallmatrix}
  1+\pi&0\\
  0&(1+\pi)^{-1}
\end{smallmatrix}\right) \right] = \left(\begin{smallmatrix}
  1&-\pi^2a(\pi+2)\\
  0&1
\end{smallmatrix}\right) \in [H:H]$.
\item $\left[\left(\begin{smallmatrix}
  1&0\\
  \pi b&1
\end{smallmatrix}\right):\left(\begin{smallmatrix}
  1+\pi&0\\
  0&(1+\pi)^{-1}
\end{smallmatrix}\right) \right] = \left(\begin{smallmatrix}
  1&0\\
  b\pi^2\tfrac{\pi+2}{(\pi+1)^2}&1
\end{smallmatrix}\right) \in [H:H]$.
\end{itemize}

\vspace{0,2cm} This shows that the first two families of generators of
$H'$ lie inside $[H:H]$.  In order to prove that $[H:H] = H'$ it only
remains to check that $\left(\begin{smallmatrix}
    1+\pi^2z&0\\
    0&(1+\pi^2z)^{-1}
\end{smallmatrix}\right) \in [H:H]$ for any $z \in \Om$. But
$\left[\left(\begin{smallmatrix}
  1&0\\
  \pi&1
\end{smallmatrix} \right):\left(\begin{smallmatrix}
  1&\pi c\\
  0&1
\end{smallmatrix} \right) \right] = \left(\begin{smallmatrix}
  1-c\pi^2&c^2\pi^3\\
  -c\pi^3&c^2\pi^4 + c\pi^2+1
\end{smallmatrix} \right) \in [H:H]$.
Multiplying this element by matrices of the form $\left(\begin{smallmatrix}
  1&0\\
  \pi^2x&1
\end{smallmatrix}\right)$ we get $\left(\begin{smallmatrix}
  1-c\pi^2&c^2\pi^3\\
  0&(1-c\pi^2)^{-1}
\end{smallmatrix}\right) \in [H:H]$
(as we can raise the power of $\pi$ appearing in the place $(2,1)$
eventually making it $0$ modulo $\pi^n$). The same argument applies to
matrices of the form $\left(\begin{smallmatrix}
    1&\pi^2y\\
    0&1
\end{smallmatrix}\right)$ we get $\left(\begin{smallmatrix}
  1-c\pi^2&0\\
  0&(1-c\pi^2)^{-1}
\end{smallmatrix}\right) \in [H:H]$.
\end{proof}

Lemma \ref{nice} will let us prove the existence of auxilliary primes that kill global obstructions. We introduced special primes
because otherwise we would have not been able to modify the behaviour of $f$.

\begin{lemma}
  Let $\rho_n$ and $P$ as before. Then there exists a finite set $P'$
  consisting of nice primes for $\rho_n$ and eventually one special
  prime for $f$ such that $\sha^1_{P\cup P'}(\Ad)$ and $\sha^2_{P\cup
    P'}(\Ad)$ are both trivial.
\end{lemma}
\begin{proof}
 If $f\notin \sha^1_P(\Ad)$ this follows from taking basis $\set{f_1, \ldots, f_r}$ and 
 $\set{\phi_1, \ldots, \phi_s}$ of $\sha^1_{P}(\Ad)$ and $\sha^2_{P}(\Ad)$
 respectively and choosing, by applying Lemma \ref{nice}, sets of nice primes $q_1, \ldots, q_r$ and $q'_1,\ldots, q'_s$ such that
 \begin{itemize}
 \item $f_i|_{G_{q_j}} = 0$ if $i \neq j$ and $f_i|_{G{q_i}} \neq 0$.
 \item $\phi_i|_{G_{q'_j}} = 0$ if $i \neq j$ and $\phi_i|_{G_{q'_i}} \neq 0$.
 \end{itemize}
 If, otherwise, $f \in \sha^1_P(\Ad)$, we do the same but taking a special prime for $f$ instead of a nice prime.
\end{proof}

From the previous lemmas, we can assume that $\sha^1_P(\Ad)$ and
$\sha^2_P(\Ad)$ are both trivial by enlarging $P$ if necessary.  This
imply, as in \cite{ravilarsen}, the following two key propositions.

\begin{prop}
\label{surj}
Let $S$ be a finite set of primes and $\rho_m:G_S\to \GL_2(\Om/\pi^m)$
a continuous representation such that $\sha^1_S(\Ad) = \sha^2_S(\Ad)=
0$.  Then there exists a set $Q$ of nice primes for $\rho_m$ such that
the map $$\coho^1(G_{S\cup Q},\Ad) \to \bigoplus_{v\in S}
\coho^1(G_v,\Ad)$$ is an isomorphism.
\end{prop}

\begin{proof}
Given the existence of auxilliary primes, this is just Lemma $8$ of \cite{ravilarsen}.
\end{proof}

\begin{prop}
\label{fix}
Let $\rho_m$, $S$ and $Q$ as in the Proposition \ref{surj}. For each
$q_i \in Q$ pick an element $h_i \in \coho^1(G_{q_i},\Ad)$.  Then
there is a finite set $T$ of nice primes for $\rho_m$ and an
element $$g \in \coho^1(G_{S\cup Q \cup T}, \Ad)$$ satisfying
\begin{itemize}
 \item $g|_{G_v} = 0$ for $v \in S$.
 \item $g(\sigma_{q_i}) = h_i(\sigma_{q_i})$ for $q_i \in Q$.
 \item $g(\sigma_v) = 0$ for $v \in T$.
\end{itemize}
\end{prop}

\begin{proof}
 This is Corollary $11$ of \cite{ravilarsen}, which follows from Lemmas $8$ and $9$ and Proposition $10$ from the same work. It
 can be easily checked that the proofs given in that paper adapt well to our setting. 
\end{proof}

We are now able to state and prove the main theorem of this section. Recall from for $v$ a nice prime
there is a set $C_v$ of deformations to characteristic $0$ and a subspace $N_v \subseteq \coho^1(G_v,\Ad)$
preserving its reductions.

\begin{thm}
\label{lifta}
Let $\rho_n$ and $P$ as in Theorem \ref{abstract} and let $\alpha$ be
an integer greater or equal than $n$. Pick, for each $v\in P$ a lift
 \[
\rho_{v,\alpha}:G_v \to \GL_2(\Om/\pi^{\alpha})
\]
of $\rho_n|_{G_v}$. Then, there is a finite set of nice primes $P'$
for $\rho_n$ and a lift
\[
\rho_{\alpha}: G_{P\cup P'} \to \GL_2(\Om/\pi^{\alpha})
\]
of $\rho_n$ such that $\rho_{\alpha}|_{G_v} \simeq \rho_{v,\alpha}$
for every $v \in P$ and $\rho_{\alpha}|_{G_v}$ is a reduction of some member
of $C_v$ for every $v \in P'$.
\end{thm}

\begin{proof}
  We will prove the theorem by induction on $\alpha$. If $\alpha = n$
  the statement is trivial. Assume the theorem is true for $\alpha =
  m$, and apply it with the collection of local deformations given by
  the reductions mod $\pi^m$ of the local representations
  $\rho_{v,m+1}$. Then, there is a lift $\rho_m:G_{P\cup P'} \to
  \GL_2(\Om/\pi^m)$ such that $\rho_m|_{G_v} = \rho_{v,m+1}
  \pmod{\pi^m}$, where $P'$ consists on nice primes for $\rho_n$. We
  will add two sets of nice primes in order to first get a lift of
  $\rho_m$ to $\Om/\pi^{m+1}$ and then locally adjust this
  lift. Since $\sha^2_P(\Ad) = 0$, $\rho_m$ has no global
  obstructions. Also observe that $\rho_m$ is unobstructed at the
  primes of $P$, as $\rho_m|_{G_v}$ lifts to $\rho_{v,m+1}$ and at the
  primes of $P'$ too, as $\rho_m|_{G_v}$ is the reduction of some
  member of $C_v$.  Therefore $\rho_m$ is both globally and locally
  unobstructed implying that it lifts to some
 $$\trho_{m+1}:G_{P\cup P'} \to \GL_2(\Om/\pi^{m+1}).$$
 To complete the proof, we need to fix the local behaviour of $\trho_{m+1}$. We will do this in two steps.
 First of all, pick for each $v \in P\cup P'$ a class $u_v \in \coho^1(G_v,\Ad)$ such that
 
 \begin{itemize}
 \item $(1+\pi^{m}u_v)\trho_{m+1}|_{G_v} \simeq \rho_{v,m+1}$ for $v \in P$.
 \item $(1+\pi^{m}u_v)\trho_{m+1}|_{G_v}$ is a reduction of a member of $C_v$ for $v \in P'$.
 \end{itemize}
 
 Now, let $Q$ be the set of nice primes produced by applying
 Proposition \ref{surj} to $S = P\cup P'$.  As the
 map $$\coho^1(G_{P\cup P'\cup Q},\Ad) \to \bigoplus_{v\in P\cup P'}
 \coho^1(G_v,\Ad)$$ is an isomorphism, there is a class $g_1 \in
 \coho^1(G_{P\cup P'\cup Q},\Ad)$ such that $g_1|_{G_v} = u_v$ for all
 $v\in P \cup P'$. Acting by this element on $\trho_{m+1}$ fixes its
 local shape at the places of $P\cup P'$ but may ruin it at the newly
 added primes of $Q$. We will solve this issue by adding a second set
 of auxilliary primes.
 
 We pick, for each $q_i \in Q$ a class $h_i \in \coho^1(G_{q_i},\Ad)$ such that 
 $$(1+\pi^m(h_i+g_1))\trho_{m+1}(\sigma_{q_i}) = \begin{pmatrix}
 q_i&0\\0&1
 \end{pmatrix}.$$
 Let $T$ and $g_2$ be respectively the set of nice primes and the element of $\coho^1(G_{P\cup P'\cup Q\cup T}, \Ad)$
 obtained from applying Proposition \ref{fix} with $S = P\cup P'$ and $Q$ and $h_i$ as above. It is easy
 to check that 
 \begin{itemize}
 \item $(1+\pi^mg)\trho_{m+1}|_{G_v} \simeq \rho_{v,m+1}$ for $v \in P$.
 \item $(1+\pi^mg)\trho_{m+1}|_{G_v} \in C_v$ for $v \in P' \cup Q \cup T$.
\end{itemize}  
 It follows that $\rho_{m+1} = (1+\pi^mg)\trho_{m+1}$ satisfies what we need, completing the proof.    
 \end{proof}

\subsection{Exponent $\alpha$ and above}

Assume we have a representation $\rho_n$ as in Theorem \ref{abstract}
with $n\geq \alpha$ (since otherwise we apply Theorem \ref{lifta}). To
ease the notation, let $P$ denote the set $P\cup P'$ if we applied
Theorem~\ref{lifta}.  

Recall from Sections $2$ and $3$ that for exponets bigger than
$\alpha$ we have defined for each $v \in P$ a set of deformations
$C_v$ of $\rho_n$ to characteristic $0$ and a subspace $N_v \subseteq
\coho^1(G_v, \Ad)$ such that $N_v$ preserve the reductions of the
elements in $C_v$, in the sense of Proposition \ref{localtypes}. They
also satisfy that $\dim N_v = \dim \coho^1(G_v, \Ad) - \dim
\coho^2(G_v, \Ad)$ for $v \in P \backslash \set{p}$ and $\dim N_p =
\dim \coho^1(G_p, \Ad) - \dim \coho^2(G_p, \Ad) + 1$.

In this setting, we can mimic the arguments of \cite{ravi3} in our
situation with some minor modifications. We
start by collecting a series of results that will prove useful.\\

\begin{lemma}
\label{dim}
 Let $r = \dim \sha_P^1((\Ad)^*)$ and $s = \sum_{v \in P} \dim \coho^2(G_v, \Ad)$. Then $$\dim \coho^1(G_P,\Ad) = r+s+2.$$ 
\end{lemma}

\begin{prop}
\label{aux2}
 Let $\set{f, f_1, \ldots, f_{r+s+1}}$ be a basis of $\coho^1(G_P, \Ad)$, where $f$ is the element attached to $\rho_n$ mod $\pi^2$. 
 There is a set $Q = \set{q_1, \ldots, q_r}$ of nice primes for $\rho_n$ not in $P$ such that:
 \begin{itemize}
  \item $\sha^1_{P\cup Q}((\Ad)^*) = 0$ and $\sha^2_{P\cup Q}(\Ad) = 0$.
  \item $f_i|_{G_{q_j}} = 0$ for $i\neq j$ and $f|_{q_j} = 0$ for all $j$.
  \item $f_i|_{G_{q_i}} \notin N_{q_i}$
  \item The inflation map $\coho^1(G_P, \Ad) \to \coho^1(G_{P\cup Q}, \Ad)$ is an isomorphism.
 \end{itemize}
\end{prop}

\begin{proof}
  Lemma \ref{dim} is just the Lemma before Lemma $10$ in \cite{ravi3}. The proof of Proposition \ref{aux2} mimics 
  that of Fact 16 of \cite{ravi3}. Observe that, in the spirit of what is remarked in the proof of Proposition \ref{nice}, 
  the condition for a prime $q$ being nice for $\rho_n$ is not compatible with $f|_{G_q} \notin
  N_q$ as every subspace $N_q$ is chosen such that $f|_{G_q} \in
  {N_q}$.  In particular, $f$ is always in the kernel of the
  restriction map
 $$ \coho^1(G_{P\cup Q}, \Ad) \to \bigoplus_{v\in P\cup Q}\coho^1(G_v, \Ad)/N_v.$$
  As we have checked in Proposition \ref{nice}, conditions for the rest of the $f_i$'s are independent from being
  $\rho_n$-nice, so the same proof as in \cite{ravi3} works.
\end{proof}

\begin{lemma}
 Let $\langle f, f_1, \ldots, f_d\rangle$ be the kernel of the map $$ \coho^1(G_{P\cup Q}, \Ad) \to \bigoplus_{v\in P}\coho^1(G_v, \Ad)/N_v.$$
 Then $r \geq d$. 
\end{lemma}
\begin{proof}
  Follows from the formulas $\dim \oplus_{v\in
    P}\coho^1(G_v, \Ad)/N_v = s+1$ and $\dim \coho^1(G_{P\cup Q}, \Ad)
  = r+s+2$.
\end{proof}

\begin{lemma}
  There is a finite set of nice primes $\set{t_{r+1}, \ldots, t_d}$ for
  $\rho_n$ such that
 \begin{itemize}
  \item $f_i|_{t_j} = 0$ if $i\neq j$ and $f_i|_{t_i} \notin N_{t_i}$.
  \item The restriction map $$\coho^1(G_{P\cup Q \cup T},\Ad) \to \bigoplus_{v \in P} \coho^1(G_v, \Ad)/N_v$$
  is surjective.
 \end{itemize}
\end{lemma}

\begin{proof}
  This is Lemma 14 of \cite{ravi3}, which relies on Proposition 10. It
  is easy to check that the proof given there adapts to the ramified
  setting, using the results we have available so far, as it does not
  involve picking nice primes that satisfy properties at $f$.
\end{proof}

The set $Q \cup T$ will serve as the auxilliary set for Theorem \ref{abstract}. So far we have the following properties:

\begin{itemize}
 \item For $1 \leq i \leq r$: $f_i|_{G_v} = 0$ for all $v\in Q \cup T$ except $q_i \in Q$ for which $f_i|_{G_{q_i}} \notin N_{q_i}$.
 \item For $r+1 \leq i \leq d$: $f_i|_{G_v} = 0$ for all $v\in Q \cup T$ except $t_i \in T$ for which $f_i|_{G_{t_i}} \notin N_{t_i}$.
 \item The restriction $\langle f_{d+1}, \ldots f_{d+s+1} \rangle \to \oplus_{v \in P} \coho^1(G_v, \Ad)/N_v$ is an isomorphism.
 \item $f|_{G_v} \in N_v$ for every $v \in P \cup Q \cup T$.
\end{itemize}

It easily follows from these properties that

\begin{prop}
\label{sobre}
 The map $$\coho^1(G_{P\cup Q\cup T},\Ad) \to \bigoplus_{v \in P\cup Q\cup T} \coho^1(G_v,\Ad)/N_v$$
 is surjective and has one dimensional kernel generated by $f$.
\end{prop}

From this, we can easily deduce Theorem \ref{abstract} in the same way
as Theorem 1 is proved in \cite{ravi3}. Moreover we have the following
result

\begin{thm}
\label{abstract2}
Let $\rho_n:G_\Q \to \GL_2(\Om/\pi^n)$ and $\rho_v:G_v \to \GL_2(\Om)$
as in Theorem \ref{abstract}. Consider the collection $\mathcal{L}$ of
deformation conditions given by the pairs $(C_v, N_v)$ for $v \in P
\cup Q \cup T$.  Then the deformation problem with fixed determinant
and local conditions $\mathcal{L}$ has universal deformation ring
$R_{u} \simeq W(\F)[[X]]$.
\end{thm}

\begin{proof}
 Proposition \ref{sobre} tells us that $\coho^1_{\mathcal{L}}(G_{P\cup Q\cup T},\Ad) = \langle f \rangle$. As $\sha^2_{P\cup Q\cup T}(\Ad) = 0$, we also 
 know that the problem is unobstructed. This proves the theorem. 
\end{proof}

\section{Modularity}

So far we have constructed, for a mod $\pi^n$ representation $\rho_n$,
which is nearly ordinary at $p$, a global deformation $\rho_u: G_{\Q}
\to \GL_2(W(\F)[[X]])$ that lifts $\rho_n$ and is also nearly ordinary
at $p$. This gives a family of lifts of $\rho_n$ to rings of dimension
and characteristic $0$. The purpose of this section is to prove that
when $\rho_n$ is \textit{ordinary} at $p$, at least one of these lifts
is modular. However, getting a lift with a previously fixed weight
seems out of the scope of this work. We will prove the following.

\begin{thm}
\label{modular}
 Let $p$ be a prime, $\Om$ the ring of integers of a finite extension $K/\Q_p$ with ramification degree $e > 1$ and $\pi$ its local uniformizer.
 Let $\rho_n: G_\Q \to \GL_2(\Om/\pi^n)$ be a continuous representation satisfying
 \begin{itemize}
  \item $\rho_n$ is odd.
  \item $\Img(\rho_n)$ contains $\SL_2(\Om/p)$ if $n > e$ and $\SL_2(\Om/\pi^n)$ otherwise.
  \item $\rho_n$ is ordinary and not scalar at $p$.
 \end{itemize}
 
 Let $P$ be a set of primes containing the set of ramification of $\rho_n$, and for each $v \in P$ pick a local deformation $\rho_v:G_v\to \GL_2(\Om)$
 lifting $\rho_n|_{G_v}$ which is not bad. Then there exists a finite set of primes $Q$ and a continuous representation 
 $\rho: G_{P\cup Q} \to \GL_2(\Om)$ such that
 
 \begin{itemize}
  \item $\rho$ lifts $\rho_n$, i.e. $\rho \equiv \rho_n \pmod{\pi^n}$.
  \item $\rho$ is modular.
  \item For every $v \in P$, $\rho|_{I_v} \simeq \rho_v|_{I_v}$ over $\GL_2(\Om)$.
  \item For every $q \in Q$, $\rho|_{I_q}$ is unipotent and $q \neq \pm 1 \pmod{p}$ for all but possibly one $q \in Q$.
  \item $\rho$ is ordinary at $p$.
 \end{itemize}
\end{thm}

\begin{proof}
  Theorem \ref{abstract2} implies the existence of an universal ring
  $R_u \simeq \Z_p[[X]]$ that parametrizes nearly ordinary
  deformations with fixed determinant $\omega\chi^k$ and satisfying
  certain local conditions at the primes of $P \cup Q$.

 Observe that for each morphism of local $\Z_p$-algebras $\gamma:R_u \to \Om$ we have a nearly ordinary deformation $\rho_\gamma$ with
 coefficients in $\Om$. If we look at the local structure at $p$ of this deformation we find that 
 $$\rho_\gamma|_{I_p} \simeq \begin{pmatrix}
                             \omega_1\psi_1\chi^b&*\\
                             0&\omega_2\psi_2\chi^{k-b}
                            \end{pmatrix}$$ for $\omega_1$, $\omega_2$ unramified characters, $\psi_1$ and $\psi_2$ of finite order
                            and $b \in \Z_p$ (using that any $p$-adic character
 can be written as a power of the cyclotomic character times a character of finite order times an unramified character). 
 Twisting $\rho_\gamma$ by $\psi_2^{-1}\chi^{b-k}$ we get, for each
 $\gamma \in \Hom_{loc-\Z_p}(R_u,\Om)$ an ordinary deformation of $\rrho$. Denote the representation 
 $\psi_2^{-1}\chi^{b-k}\rho_\gamma$ by $\trho_\gamma$.

 As we are asking $\rrho$ to be modular, the work of \cite{dtaylor}
 ensures that $R_u$ contains a twist of characteristic zero modular
 points. This is, there is a morphism $\gamma_k: R_u \to
 \overline{\Q_p}$ such that the twisted deformation $\trho_{\gamma_k}$
 is ordinary of weight $k$ (and therefore modular). Specifically,
 \cite{dtaylor} guarantees that $\rrho$ has an ordinary lift of
 classical weight and arbitary determinant. After twisting this lift
 by the corresponding power of the cyclotomic character, it lies in
 our family of deformations with fixed determinant. This implies that
 the family of representations that we have constructed is part of the
 Hida family of $\trho_{\gamma_k}$. Let $\mathcal{H}$ be this Hida
 family, then there is a morphism $\Omega: \Hom_{loc-\Z_p}(R_u,\Om)
 \to \mathcal{H}$. As $\rho_n$ admits different lifts to $\Om$, this
 morphism is not constant.  Recall this Hida family is equipped with
 his corresponding weight map $w:H\to W$.
 
 On the other hand, we have a morphism $\theta:\Z_p[X] \to \Om/\pi^n$ which induces $\rho_n$. We know that any morphism 
 $\tilde{\theta}: \Z_p[[X]] \to \Om$ that lifts $\theta$ induces a lift $\rho_{\tilde{\theta}}$ of $\rho_n$ to $\Om$. But 
 $\Hom_{loc-\Z_p}(\Z_p[[X]],\Om) \simeq \mathcal{M}_\Om$ (as defining such a morphism ammounts to choosing an element of the maximal
 ideal of $\Om$ where to send $X$) and the set of morphisms $\tilde{\theta}$ that lift $\theta$ correspond to an open set 
 $\mathcal{U} \subseteq \mathcal{M}_{\Om}$ under this identification. 
 
 Overall, if we restrict the previously constructed morphism we have
 $\Omega|_{\mathcal{U}}:\mathcal{U} \to \mathcal{H}$. The image of
 this morphism must be an open set (as $\Omega$ is analytic), and if
 we compose it with the weight map we get an open set
 $w\circ\Omega(\mathcal{U}) \subseteq W$.  This open set necessarily
 contains a classical point, and any pre-image of this classical point
 gives raise to an ordinary lift of $\rho_n$ of integer weight which
 we call $\rho$. Finally, by the main theorem of \cite{sw}, $\rho$ is
 modular.
\end{proof}

\begin{coro}
\label{cor}
 Let $\rho_n$ in the same hyphoteses as in Theorem \ref{modular}. Then there exists a finite set of primes $Q$ and a continuous representation 
 $\rho: G_{P\cup Q} \to \GL_2(\Om)$ satisfying the consequences of Theorem \ref{modular} which also is of weight $2$.
\end{coro}

\begin{proof}
 Once we have a lift of $\rho_n$ to characteristic $0$ (which exists by Theorem \ref{modular}) the corollary follows from the results of Section $5$ of 
 \cite{KR}. This gives another lift of $\rho_n$ which may ramify at a bigger set of primes but is of weight $2$.
\end{proof}

\bibliographystyle{alpha}
\bibliography{biblio}

\end{document}